\documentclass[a4paper]{amsart}
\usepackage{amssymb}
\usepackage{amscd}
\usepackage{amsthm}
\usepackage{amsmath}
\usepackage{mathtools}
\usepackage{latexsym}
\usepackage[all]{xy}
\usepackage[utf8]{inputenc}
\usepackage{enumitem}

\usepackage{tikz-cd}

\setlist[enumerate]{label={(\roman*)}}

\theoremstyle{plain}
\newtheorem{theorem}{Theorem}
\newtheorem{corollary}[theorem]{Corollary}
\newtheorem{lemma}[theorem]{Lemma}
\newtheorem{proposition}[theorem]{Proposition}

\theoremstyle{definition}
\newtheorem{definition}[theorem]{Definition}

\theoremstyle{remark}

\newtheorem{remark}{Remark}

\numberwithin{theorem}{section}
\usepackage{newtxtext}
\usepackage{newtxmath}

\DeclareMathAlphabet\urwscr{U}{urwchancal}{m}{n}%
\DeclareMathAlphabet\rsfscr{U}{rsfso}{m}{n}
\DeclareMathAlphabet\euscr{U}{eus}{m}{n}
\DeclareFontEncoding{LS2}{}{}
\DeclareFontSubstitution{LS2}{stix}{m}{n}
\DeclareMathAlphabet\stixcal{LS2}{stixcal}{m} {n}

\newcommand{\card}{\mbox{\rm{card\,}}}

\newcommand{\add}{\mbox{\rm{add\,}}}

\newcommand{\Def}{\mbox{\rm{Def\,}}}

\newcommand{\Filt}{\mbox{\rm{Filt\,}}}

\newcommand{\End}{\mbox{\rm{End\,}}}

\newcommand{\Hom}[3]{\operatorname{Hom}_{#1}(#2,#3)}

\newcommand{\rmod}[1]{\mbox{\rm{Mod}--}{#1}}

\newcommand{\lmod}[1]{{#1}\mbox{--\rm{Mod}}}
\newcommand{\ModR}{\text{Mod-}R}

\begin{document}
	
	\title{Flat relative Mittag-Leffler modules and approximations}
	
	\author{Asmae Ben Yassine and Jan Trlifaj}
	
	\begin{abstract} The classes $\mathcal D _{\mathcal Q}$ of flat relative Mittag-Leffler modules are sandwiched between the class $\mathcal F \mathcal M$ of all flat (absolute) Mittag-Leffler modules, and the class $\mathcal F$ of all flat modules. Building on the works of Angeleri H\" ugel, Herbera, and \v Saroch, we give a characterization of flat relative Mittag-Leffler modules in terms of their local structure, and show that Enochs' Conjecture holds for all the classes $\mathcal D _{\mathcal Q}$. In the final section, we apply these results to the particular setting of f-projective modules.
	\end{abstract}
	
	\date{\today}
	
  \thanks{Research supported by GA\v CR 20-13778S and SVV-2020-260589}
	
			
\maketitle

\section*{Introduction}
	
For a ring $R$, denote by $\mathcal P$, $\mathcal F$, and $\mathcal F \mathcal M$ the classes of all projective, flat, and flat Mittag-Leffler  (right $R$-) modules, respectively. We always have the inclusions $\mathcal P \subseteq \mathcal F \mathcal M \subseteq \mathcal F$. The equality $\mathcal P = \mathcal F \mathcal M = \mathcal F$ holds, if and only if $R$ is a right perfect ring. By a classic result of Bass, in this case $\mathcal P$ is a covering class consisting of modules isomorphic to direct sums of (indecomposable projective) modules generated by primitive idempotents of the ring $R$.

If  $R$ is not right perfect, then $\mathcal P \subsetneq \mathcal F \mathcal M \subsetneq \mathcal F$. In fact, though the classes $\mathcal P$ and $\mathcal F \mathcal M$ contain the same countably generated modules, there always exist $\aleph_1$-generated modules in $\mathcal F \mathcal M$ that are not projective, cf.\ \cite[\S VII.1]{EM}. Moreover, there exist countably presented modules $N \in \mathcal F \setminus \mathcal F \mathcal M$. Each such module $N$ is called a Bass module \cite{ST}.

By a classic theorem of Kaplansky, each projective module is a direct sum of countably generated projective modules, so the class $\mathcal P$ is $\aleph_1$-decomposable. If $\kappa = \card R + \aleph_0$, then each flat module is known to be a transfinite extension of $\leq \kappa$-presented flat modules, so the class $\mathcal F$ is $\kappa^+$-deconstructible. The class $\mathcal P$ is easily seen to be precovering, while $\mathcal F$ is a covering class by \cite{BEE} (see Section \ref{prelim} for unexplained terminology).   

The intermediate class $\mathcal F \mathcal M$ can be described as the class of all {\lq}locally projective{\rq}, or better $\aleph_1$-projective modules \cite{HT}. Its global structure over non-right perfect rings is known to be quite complex: there is no cardinal $\lambda$ such that $\mathcal F \mathcal M$ is $\lambda$-deconstructible \cite{HT}; moreover, the class $\mathcal F \mathcal M$ is not precovering \cite{S}.

\medskip
In this note, we will deal with classes of flat relative Mittag-Leffler modules, or more precisely, flat $\mathcal Q$-Mittag-Leffler modules for a class of left $R$-modules $\mathcal Q$. 

The notion of an (absolute) Mittag-Leffler module was introduced already in the seminal paper by Raynaud and Gruson \cite{RG}, and studied in a number of sequel works revealing its many facets. Relative Mittag-Leffler modules appeared much later, in the Habilitationsschrift of Rothmaler \cite{R1}. Rothmaler has further pursued the model theoretic point of view in \cite{R2}, where he proved that if $\mathcal Q$ is a definable class of left $R$-modules and $D(\mathcal Q)$ is its dual definable class of (right $R$-) modules, then $\mathcal Q$-Mittag-Leffler modules are exactly the $D(\mathcal Q)$-atomic modules, \cite[Theorem 3.1]{R2}. For a very recent application of relative atomic modules to a description of Ziegler spectra of tubular algebras, we refer to \cite{P}. 

A detailed algebraic study of relative Mittag-Leffler modules was performed in \cite{AH} (see also \cite{H} and \cite{HT}); notably, it discovered their role in (infinite dimensional) tilting theory.  

Following \cite{HT}, we denote the class of all flat $\mathcal Q$-Mittag-Leffler modules by $\mathcal D _{\mathcal Q}$. Thus $\mathcal F \mathcal M \subseteq \mathcal D _{\mathcal Q} \subseteq \mathcal F$. Notice that $\mathcal D _{\lmod R} = \mathcal F \mathcal M$ and $\mathcal D _{\{ 0 \}} = \mathcal F$, while $\mathcal D _{\{ R \}}$ is the class of all f-projective modules studied by Goodearl et al. in \cite{F}, \cite{G}, etc.

Our goal here is to investigate the structure and approximation properties of the class $\mathcal D _{\mathcal Q}$ in dependence on $\mathcal Q$. In Theorem \ref{test}, we prove that the classes $\mathcal D _{\mathcal Q}$ are determined by their countably presented modules, while Theorem \ref{two} shows that approximation properties of $\mathcal D _{\mathcal Q}$ depend completely on whether there exists a Bass module $N \notin \mathcal D _{\mathcal Q}$. In the final part, we apply these results to the particular setting of $\mathcal Q = \{ R \}$, i.e., to the f-projective modules.

\section{Preliminaries}\label{prelim}

For a ring $R$, we denote by $\rmod R$ the class of all (right $R$-) modules, and by $\lmod R$ the class of all left $R$-modules.
   \subsection{Filtrations and deconstructible classes}
   Let $ R $ be a ring, $ M $ a module, and $ \mathscr{C} $ a class of modules. A family of submodules, $\mathcal{M}=(M_{\alpha} \mid \alpha \leq \sigma)$, of $ M $ is called a \textit{continuous chain} in $M$, provided that $M_{0}=0$, $M_{\alpha} \subseteq M_{\alpha+1}$ for each $\alpha < \sigma $, and $M_\alpha = \bigcup_{\beta < \alpha} M_{\beta}$ for each limit ordinal $\alpha \leq \sigma$.\\
   A continuous chain $\mathcal{M} $ in $M$ is a $ \mathscr{C} $-\textit{filtration} of $ M $, provided that $M=M_{\sigma}$, and each of the modules $M_{\alpha+1}/M_{\alpha}\;(\alpha < \sigma)$ is isomorphic to an element of $ \mathscr{C} $.\\
   $M$ is called $ \mathscr{C} $-\textit{filtered}, provided that $M$ possesses at least one $ \mathscr{C} $-filtration. We will use the notation $ \Filt(\mathscr{C}) $ for the class of all $ \mathscr{C} $-filtered modules. The modules $M \in \Filt(\mathscr{C})$ are also called \textit{transfinite extensions} of the modules in $ \mathscr{C} $. A class $\mathscr{A}$ is said to be \textit{closed under transfinite extensions} provided that $\mathscr{A} = \Filt(\mathscr{A})$. Clearly, this implies that $ \mathscr{A} $ is closed under extensions and arbitrary direct sums. \\
    Given a class $\mathscr{C}$ and a cardinal $  \kappa $, we use $\mathscr{C}^{\leq \kappa}$ and $\mathscr{C}^{< \kappa}$ to denote the subclass of $\mathscr{C}$ consisting of all $\leq \kappa$-presented and $< \kappa$-presented modules, respectively.\\
    Let $ \kappa $ be an infinite cardinal. A class of modules $ \mathscr{C} $ is $ \kappa $-\textit{deconstructible}
    provided that $\mathscr{C}  \subseteq \Filt(\mathscr{C}^{< \kappa})$. If moreover each module $M \in \mathscr{C}$ is a direct sum of modules from $\mathscr{C}^{< \kappa}$, then $ \mathscr{C} $ is called $ \kappa $-\textit{decomposable}. 
		For example, the class $ \mathscr{P} $ of all projective modules is $ \aleph_{1} $-deconstructible by a classic theorem of Kaplansky. A class $ \mathscr{C} $ is \textit{deconstructible} in case it is $ \kappa $-deconstructible for some infinite cardinal $ \kappa $.
    \subsection{Approximations}
    	A map $ f \in \Hom R {C} {M} $ with $ C \in \mathscr{C}  $ is a \emph{$ \mathscr{C} $-precover} of $ M $, if the abelian group homomorphism $ \Hom R {C'} {f}:\Hom R {C'}{C} \rightarrow \Hom R {C'}{M}$ is surjective for each $ C' \in \mathscr{C} $.\\
    	A $ \mathscr{C} $-precover $ f \in \Hom R {C} {M }$ of $ M $ is called a \emph{$ \mathscr{C} $-cover} of $ M $, provided that $ f $ is right minimal, that is, provided $fg=f$ implies that $ g $ is an automorphism for each $ g \in \End_{R}(\mathscr{C}) $.\\
    	$ \mathscr{C} \subseteq \ModR $ is a \emph{precovering class} (\emph{covering class}) provided that each module has a $ \mathscr{C}$-precover ($\mathscr{C} $-cover).
    \subsection{(Relative) Mittag-Leffler modules}
    Let $ R $ be a ring. A module $ M $ is \textit{Mittag-Leffler} provided that the canonical group homomorphism
    $$ \varphi: M \otimes_R \prod_{i \in I} N_{i}\rightarrow \prod_{i \in I} M \otimes_R N_{i}$$
    defined by
    $$ \varphi(m \otimes_R (n_{i})_{i \in I})= (m \otimes_R n_{i})_{i \in I}$$
    is monic for each family $ (N_{i}\;|\; i \in I) $ of left $ R $-modules. \\
    Let $M \in \rmod R$ and $\mathcal Q \subseteq \lmod R$. Then $M$ is \emph{$\mathcal Q$-Mittag-Leffler}, provided that the canonical morphism $M \otimes {\prod_{i \in I} Q_i} \to \prod_{i \in I} M \otimes_R Q_i$ is injective for any family $( Q_i \mid i \in I )$ consisting of elements of $\mathcal Q$.
    \subsection{Direct limits and add(M)}
    Let $ \mathscr{C} $ be any class of modules and $ \mathscr{D}=(C_{i}, f_{ji} \mid i \leq j \in I) $ a direct system of modules in $ \mathscr{C} $. Viewing $ \mathscr{D} $ as a diagram in the category $\rmod R$, we can form its colimit, $ (M, f_{i} \mid i \in I) $. In particular, $ M $ is a module, and $ f_{i}\in \Hom R{M_{i}} {M} $ satisfy $ f_{i} = f_{ij}f_{j} $ for all $ i\leq j \in I $.\\
    This colimit (or sometimes just the module $ M $ itself) is called the \textit{direct limit} of the direct system $ \mathscr{D} $. It is denoted by $ \varinjlim_{i \in I} M_{i} $ (or just $ \varinjlim \mathscr{D} $).\\ 
    Let $\mathcal Q \subseteq \lmod R$. We denote by $ \varinjlim \mathcal Q $ the class of all modules $ N $ such that $N=\varinjlim Q_{i}  $ where $ (Q_{i}, f_{ji} \mid i \leq j \in I) $ is a direct system of modules from $ \mathcal Q $.\\
    Let $ R $ be a ring, $ M $ be a module. We define $ \add(M) $ to be the class of all modules isomorphic to direct summands of finite direct sums of copies of $M$. 
    \subsection{Bass modules}
    Given a class $ \mathscr{C} $ of finitely generated free modules, we call a module $ M $ a \textit{Bass module} over $\mathscr{C}$, provided that $ M $ is the direct limit of a direct system $$ C_{0} \xrightarrow{f_{0}} C_{1} \rightarrow \cdots \xrightarrow{f_{n-1}} C_{n}\xrightarrow{f_{n}} C_{n+1} \xrightarrow{f_{n+1}} \cdots$$ where $ C_{n} \in \mathscr{C}  $ for each $n < \omega$. \\ 
		If $ \mathscr{C} $ is the class of all finitely generated free modules, than the Bass modules over $ \mathscr{C} $ are just called the (unadorned) \emph{Bass modules}; they are exactly the countably presented flat modules.\\
    \\
    For basic properties of the notions defined above, we refer the reader to \cite{GT}.

\section{Flat relative Mittag-Leffler modules}
We record the following well-known properties of the class $\mathcal D _{\mathcal Q}$ of all flat $\mathcal Q$-Mittag-Leffler modules (cf.\ \cite[\SS 1 and 5]{AH}, \cite[\S 4]{HT} or \cite[3.20(a)]{GT}):

\begin{lemma}\label{prop-ML} Let $\mathcal Q \subseteq \lmod R$.
\begin{enumerate} 
\item The class $\mathcal D _{\mathcal Q}$ is closed under pure submodules, extensions, and unions of pure chains. Hence $\mathcal D _{\mathcal Q}$ is closed under transfinite extensions. 
\item $\mathcal D _{\mathcal Q}$ is a resolving subcategory of $\rmod R$ (i.e., $\mathcal D _{\mathcal Q}$ contains all projective modules, and it is closed under extensions and kernels of epimorphisms).
\end{enumerate}
\end{lemma}

\begin{remark} Clearly, $\mathcal D _{\mathcal Q}$ is closed under direct limits, iff $\mathcal D _{\mathcal Q} = \mathcal F$. This case will be examined in more detail in Theorems \ref{test}(ii) and \ref{two} below.  The closure of the class $\mathcal D _{\mathcal Q}$ under products was studied in \cite[\S4]{HT}: if $\mathcal Q$ is the limit closure of a class of finitely presented left $R$-modules, then $\mathcal D _{\mathcal Q}$ is closed under products, iff $R^R \in \mathcal D _{\mathcal Q}$ (see \cite[Theorem 4.6]{HT}). 
\end{remark}

Another basic property of the classes $\mathcal D _{\mathcal Q}$ is that in their study, one can restrict to definable classes of left $R$-modules. Recall that a class of modules is \emph{definable} provided that it is closed under direct limits, direct products and pure submodules. For each class of left $R$-modules $\mathcal Q$ there is a least definable class $\Def (\mathcal Q)$ in $\lmod R$ containing $\mathcal Q$; it is obtained by closing $\mathcal Q$ first by direct products, then direct limits, and finally by pure submodules, cf.\ \cite[Lemma 2.9 and Corollary 2.10]{H}.

\begin{lemma}\label{lim-ML} Let $\mathcal Q \subseteq \lmod R$. Then $\mathcal D _{\mathcal Q} = \mathcal D _{\Def (\mathcal Q)}$. 
\end{lemma}

Definable classes are parametrized by the subset of the set of all indecomposable pure-injective modules which they contain. So though $\lmod R$ is a proper class, there is only a set of classes of relative Mittag-Leffler modules. Note however, that it may still happen that $\mathcal D _{\Def (\mathcal Q)} = \mathcal D _{\Def (\mathcal Q ^\prime)}$ even if $\Def (\mathcal Q) \neq \Def (\mathcal Q ^\prime)$: just take a right noetherian ring $R$ which is not completely reducible, and consider the following two definable classes of left $R$-modules: $\mathcal Q = \{ 0 \}$ and $\mathcal Q ^\prime = \mathcal F ^\prime$ (the class of all flat left $R$-modules). Then $\mathcal D _{\mathcal Q} = \mathcal D _{\mathcal Q ^\prime} = \mathcal F$ by Proposition \ref{observe}(i) below. In Theorem \ref{test}, we will give a different parametrization of the classes $\mathcal D _{\mathcal Q}$, by their countably presented modules.   

Our next prerequisite was proved in \cite[2.2]{HT} (see also \cite[3.11]{GT}):

\begin{lemma}\label{qml}
Let $ R $ be a ring, $ M $ be a module, and $\mathcal Q$ be a class of left $R$-modules. Assume that $ M = \varinjlim_{\alpha \in L} M_{\alpha} $ where $ (M_{\alpha}, f_{\beta \alpha}\;|\;\alpha < \beta \in L)$ is a direct system of $\mathcal Q$-Mittag-Leffler modules. Moreover, assume that $ M'=\varinjlim_{n < \omega} M_{\alpha_{n}} $ is $\mathcal Q$-Mittag-Leffler for each countable chain $ \alpha_{0} < \cdots < \alpha_{n} < \alpha_{n+1} < \cdots$ in $ L $. 
		
Then $ M $ is $\mathcal Q$-Mittag-Leffler.
\end{lemma}

For all rings, flat relative Mittag-Leffler modules include the flat (absolute) Mittag-Leffler modules, and for some rings, even all the flat modules (see Section \ref{fproject} below). So the following description of the local structure of flat relative Mittag-Leffler modules extends simultaneously the {\lq}local projectivity{\rq} of flat Mittag-Leffler modules from \cite[Theorem 2.10(i)]{HT} and the deconstructibility, and hence abundance of small pure flat submodules, of flat modules from \cite[Lemma 6.17 and Theorem 7.10]{GT} (cf.\ also  \cite[Theorem 5.1]{AH}, \cite[Theorem 2.6]{HT}, \cite[Corollary 6.5]{R2}, and \cite[Lemma 2.5]{SS}):    

\begin{proposition}\label{one}
	Let $ R $ be a ring, $ M $ be a module, and $\mathcal Q$ be a class of left $R$-modules. Let $ \kappa= \card(R) + \aleph_0 $. Then the following conditions (i)-(iv) are equivalent:
	\begin{enumerate}
		\item $ M $ is a flat $\mathcal Q$-Mittag-Leffler module.
		\item For each subset $ C $ of $ M $ of cardinality $ \leq \kappa $, there exists a pure flat $\mathcal Q$-Mittag-Leffler submodule $N$ of $M$ such that $C \subseteq N$, and $N$ has cardinality $ \leq \kappa $.
		\item There exists a system $ \mathcal{S} $ consisting of pure flat $\mathcal Q$-Mittag-Leffler submodules of $ M $ of cardinality $ \leq \kappa $, such that for each subset $ C $ in $ M $ of cardinality $ \leq \kappa $ there is $N \in \mathcal{S} $ containing $C$, and $ \mathcal{S} $ is closed under unions of well-ordered chains of length $ \leq \kappa $.
		\item $M$ is a directed union of a direct system $ \mathcal{T} $ consisting of flat $\mathcal Q$-Mittag-Leffler submodules of $M$, such that $ \mathcal{T} $ is closed under unions of countable chains. \\
		
Consider also the following condition:	

\item There exists a system $\mathcal{U}$ consisting of countably presented flat $\mathcal Q$-Mittag-Leffler submodules $N$ of $M$ such that the inclusion $N \hookrightarrow M$ remains injective when tensored by any left $R$-module $Q \in \mathcal Q$, and satisfying the following two conditions: (a) for each countable subset $C$ in $M$ there is $N \in \mathcal{U}$ containing $C$, and (b) $\mathcal{U}$ is closed under unions of countable chains.\\

Then (v) implies (i), and if $R \in \mathcal Q$, then (v) is equivalent to (i).	 
	\end{enumerate}
\end{proposition}

\begin{proof}
	(i) $ \Rightarrow $ (ii). Let $ C $ be a subset of $ M $ with $\card C \leq \kappa $. Then there is a pure submodule $ P \subseteq^{*} M $ such that $ C \subseteq P $ and $\card P \leq \kappa $ (see e.g.\ \cite[2.25(i)]{GT}). By Lemma \ref{prop-ML}, $P$ is flat and $\mathcal Q$-Mittag-Leffler, whence (ii) holds. \\
  (ii) $ \Rightarrow $ (iii). We will prove that the set $\mathcal{S} $ consisting of all pure flat $\mathcal Q$-Mittag-Leffler submodules of $M$ of cardinality $\leq \kappa$ has the required two properties. The first one is just a restatement of (ii). For the second (closure under unions of well-ordered chains of length $\leq \kappa$), let $( N_{\alpha} \mid \alpha < \kappa )$ be a such a chain in $\mathcal{S} $. 
Let $ N= \bigcup_{\alpha < \kappa} N_{\alpha} $. Since $ N_{\alpha} $ is pure in $M$ for each $\alpha < \kappa$, $N$ is pure in $M$, too, by \cite[Lemma 2.25(d)]{GT}. Since $ \card N \leq \kappa $, (ii) yields existence of $N' \in \mathcal{S} $ such that $ N \subseteq N'$. Finally, $ N \subseteq^{*} M $ implies $ N \subseteq^{*} N' $. As $N'$ is flat and $\mathcal Q$-Mittag-Leffler, Lemma \ref{prop-ML}(i) gives $N \in \mathcal S$.\\
	(iii) $ \Rightarrow $ (iv) This is clear - just take $\mathcal{T}=\mathcal{S} $.\\
	(iv) $ \Rightarrow $ (i) First, $M$, being a directed union of flat modules, is flat. In view of (iv), we can apply Lemma \ref{qml} to the presentation of $M$ as the directed union of the elements of $ \mathcal{T} $; thus, $M$ is $\mathcal Q$-Mittag-Leffler.
	
	Assume (v). Then $M$ is a directed union of the modules in $\mathcal U$, and Lemma \ref{qml} applies, showing that $M$ is a flat $\mathcal Q$-Mittag-Leffler module.  
	
Finally, let $R \in \mathcal Q$. Assume $M$ is a flat $\mathcal Q$-Mittag-Leffler module, and let $\mathcal U$ be the set consisting of all countably presented flat $\mathcal Q$-Mittag-Leffler submodules $N$ of $M$ such that the inclusion $N \hookrightarrow M$ remains injective when tensored by any left $R$-module $Q \in \mathcal Q$. Since $R \in \mathcal Q$, the implication (1) $\Rightarrow$ (4) of 
\cite[Theorem 5.1]{AH} (for $\mathcal S$ = the class of all finitely generated free modules) yields condition (a). Let $N^\prime$ be the union of a countable chain $( N_i \mid i < \omega )$ of modules from $\mathcal U$. Then $N^\prime$ is flat, and each finite subset of $N^\prime$ is contained in some term of the chain, so by the implication (4) $\Rightarrow$ (1) of \cite[Theorem 5.1]{AH}, $N^\prime$ is a $\mathcal Q$-Mittag-Leffler module. Since the inclusion $\nu : N^\prime \hookrightarrow M$ is a direct limit of the inclusions $\nu_i : N_i \hookrightarrow M$ ($i < \omega$), $\nu \otimes_R Q$ is the direct limit of the monomorphisms $\nu_i \otimes_R Q: N_i \otimes _R Q \hookrightarrow M \otimes _R Q$ ($i < \omega$), hence $\nu \otimes_R Q$ is injective, for each $Q \in \mathcal Q$. Thus $N^\prime \in \mathcal U$, and condition (b) holds, too. 
\end{proof} 

\begin{remark}\label{nopure} 1. If $ R \notin \mathcal Q $, then (i) need not imply (v). For a simple counter-example, consider a von Neumann regular ring $ R $ such that there exists a simple module $M$ which is not countably presented (i.e., $ M = R/I $ where $ I $ is a maximal right ideal of $ R $ which is not countably generated). Examples of such rings $R$ include infinite products of fields, or endomorphism rings of infinite dimensional linear spaces. Let $\mathcal Q = \{ 0 \}$. Then $ M $ is flat (= flat $ \mathcal Q $-Mittag-Leffler), as all modules over von Neumann regular rings are flat, but (v) fails, because the only countably presented submodule of $ M $ is $ 0 $.  

2. If we remove the assumption of flatness from conditions (i)-(v), then the result still holds true, cf.\ \cite[Theorem 2.6]{HT}.

3. The system $\mathcal S$ in (iii) consists of {\lq}big{\rq} (= of cardinality $\leq \kappa$) pure submodules of $M$ and it is closed under unions of {\lq}long{\rq} (= of length $\leq \kappa$) well-ordered chains, while the system $\mathcal U$ in (v) consists of {\lq}small{\rq} (= countably presented), but possibly non-pure, submodules of $M$, and it is closed under unions of {\lq}short{\rq} (= countable) chains. 

It may even happen that no non-zero module in the system $\mathcal U$ is pure in $M$: for an example, consider the polynomial ring $R = \mathbb C [x]$, let $\mathcal Q = \{ R \}$, and let $M$ be the quotient field of $R$ viewed as an (uncountably generated) $R$-module. Then $M \in \mathcal D _{\mathcal Q}$, but $M$ has no non-zero countably generated pure submodules. So in this setting, there are only the trivial choices for a system $\mathcal S$ satisfying condition (iii) (namely, $\mathcal S = \{ M \}$, and $\mathcal S = \{ 0, M \}$), while $\mathcal U$ from (v) must be uncountable - e.g., $\mathcal U$ can be taken as the set of all countably generated submodules of $M$.  
\end{remark}

We arrive at a simple test of coincidence for various classes $\mathcal D _{\mathcal Q}$ -- one only has to check their countably presented modules:

\begin{theorem}\label{test}
Let $R$ be a ring.
\begin{enumerate} 
\item Let $\mathcal Q$ and $\mathcal Q ^\prime$ be classes of left $R$-modules containing $R$. Then $\mathcal D _{\mathcal Q} = \mathcal D _{\mathcal Q ^\prime}$, iff $\mathcal D _{\mathcal Q}$ and $\mathcal D _{\mathcal Q ^\prime}$ contain the same countably presented modules.
\item Let $\mathcal Q$ be an arbitrary class of left $R$-modules. Then $\mathcal D _{\mathcal Q} = \mathcal F$, iff each countably presented flat module is $\mathcal Q$-Mittag-Leffler. 
\item Let $\mathcal Q$ be a class of left $R$-modules containing $R$. Then $\mathcal D _{\mathcal Q} = \mathcal F \mathcal M$, iff each countably presented flat $\mathcal Q$-Mittag-Leffler module is projective. 
\end{enumerate}
\end{theorem}
\begin{proof} (i) Assume there is a module $M \in \mathcal D _{\mathcal Q} \setminus \mathcal D _{\mathcal Q ^\prime}$. Consider the system $\mathcal U \subseteq \mathcal D _{\mathcal Q}$ provided by Proposition \ref{one}(v) for the class $\mathcal Q$. Then $M$ is the directed union of the modules in $\mathcal U$, but $M \notin \mathcal D _{\mathcal Q ^\prime}$. By Lemma \ref{qml}, there is a (countably presented) module $N \in \mathcal U$ such that $N \notin \mathcal D _{\mathcal Q ^\prime}$.      

(ii) Assume there is a module $M \in \mathcal F \setminus \mathcal D _{\mathcal Q}$. Being flat, $M$ is a direct limit of a direct system $\mathfrak D$ of finitely generated free modules. By Lemma \ref{qml}, there exists a Bass module $N$ over $\mathfrak D$ such that $N \notin \mathcal D _{\mathcal Q}$. 

(iii) is a special instance of (i) for $\mathcal Q ^\prime = \lmod R$, since countably presented flat Mittag-Leffler modules are projective.
\end{proof}

Note that in \cite[Theorem 7.1]{R2}, countably generated $\mathcal Q$-Mittag-Leffler modules were characterized using $D(\Def(\mathcal Q))$-pure chains of finitely presented modules. 

\medskip
Our next theorem says that precovers (right approximations) by modules in the class $\mathcal D _{\mathcal Q}$ exist only in the threshold case of $\mathcal D _{\mathcal Q} = \mathcal F$. The theorem also confirms Enochs' Conjecture (that covering classes of modules are closed under direct limits) for all the classes $\mathcal D _{\mathcal Q}$:

\begin{theorem}\label{two}
	Let $ R $ be a ring and $\mathcal Q$ be a class of left $R$-modules. Then the following conditions are equivalent:
	\begin{enumerate}
	  \item Each Bass module (= countably presented flat module) is $\mathcal Q$-Mittag-Leffler. 
		\item $\mathcal D _{\mathcal Q} = \mathcal F$.
		\item $\mathcal D _{\mathcal Q}$ is covering.
		\item $\mathcal D _{\mathcal Q}$ is precovering.
		\item $\mathcal D _{\mathcal Q}$ is deconstructible.
		\item $\mathcal D _{\mathcal Q}$ is closed under direct limits.
	\end{enumerate}	
\end{theorem}
\begin{proof}
   (i) $ \Rightarrow $ (ii) by Theorem \ref{test}(ii).
	
	(ii) $ \Rightarrow $ (iii), (v), and (vi): This follows from the fact that for any ring $R$, the class of all flat modules is a deconstructible covering class closed under direct limits, cf.\ \cite{BEE}.
	
	(iii) $ \Rightarrow $ (iv) is trivial.
	
	(iv) $ \Rightarrow $ (i): Assume (i) fails, so there is a Bass module $N \in \mathcal F \setminus \mathcal D _{\mathcal Q}$. Let $f : A \to N$ be a (surjective) $\mathcal D _{\mathcal Q}$-precover of $N$ and $M = \ker f$. Let $\kappa$ be an infinite cardinal such that $\card R \leq \kappa$ and $\card M \leq 2^\kappa = \kappa^\omega$. By \cite[Lemma 5.6]{ST}, there are a {\lq}tree module{\rq} $L$ and an exact sequence $0 \to D \to L \to N^{(2^\kappa)} \to 0$, where $D$ is a direct sum of $\kappa$ finitely generated free modules and $L$ is flat and Mittag-Leffler. Proceeding as in the proof of \cite[Lemma 3.2]{S}, we infer that $f$ splits, whence $N \in\mathcal D _{\mathcal Q}$, a contradiction.  
	
		(v) $ \Rightarrow $ (i): This has been proved in \cite[Corollary 7.2(ii)]{HT}.
		
		(vi) $ \Rightarrow $ (i): This holds since $\mathcal F = \varinjlim \mathcal P$, whence $\varinjlim \mathcal D _{\mathcal Q} = \mathcal F$.	
\end{proof}

\begin{remark}\label{rem} If $\mathcal Q$ is a class of left $R$-modules such that $\mathcal D _{\mathcal Q} = \mathcal F$, then $\mathcal D _{\mathcal Q} = \Filt \mathcal D _{\mathcal Q} ^{\leq \kappa}$ for any infinite cardinal $\kappa \geq \card R$. 

In constrast, if $\mathcal Q$ is a class of left $R$-modules such that $\mathcal D _{\mathcal Q} \neq \mathcal F$, then the classes $\Filt \mathcal D _{\mathcal Q} ^{\leq \kappa}$, where $\kappa$ runs over all infinite cardinals $\geq \card R$, form a strictly increasing {\lq}chain{\rq} -- a proper class of subclasses of $\mathcal D _{\mathcal Q}$ -- consisting of classes closed under transfinite extensions, whose union is $\mathcal D _{\mathcal Q}$.  

Indeed, the existence of a Bass module $N \in \mathcal F \setminus \mathcal D _{\mathcal Q}$ makes it possible to construct for each infinite cardinal $\kappa \geq \card R$ a $\kappa^+$-generated flat Mittag-Leffler module $M_{\kappa^+}$ such that $M_{\kappa^+}$ is not $\mathcal D _{\mathcal Q} ^{\leq \kappa}$-filtered, cf.\ \cite[\S 5]{HT}. Thus $\Filt \mathcal D _{\mathcal Q} ^{\leq \kappa} \subsetneq \Filt \mathcal D _{\mathcal Q} ^{\leq \kappa ^+} \subsetneq \mathcal D _{\mathcal Q}$, and $\mathcal D _{\mathcal Q} = \bigcup_{\kappa \geq \card R} \Filt \mathcal D _{\mathcal Q} ^{\leq \kappa}$. Moreover, though $\mathcal F \mathcal M \subseteq \mathcal D _{\mathcal Q}$, there is no cardinal $\kappa$ such that $\mathcal F \mathcal M \subseteq \Filt \mathcal D _{\mathcal Q} ^{\leq \kappa}$.
\end{remark}

\section{$f$-projective modules}\label{fproject}

In this section, we will consider a particular kind of relative Mittag-Leffler modules, the f-projective ones. Their original definition is as follows:

\begin{definition}\label{f-proj}	
A module $ M $ is said to be \textit{f-projective} if for every finitely generated submodule $ C $ of $ M $, the inclusion map factors through a (finitely generated) free module $ F $.
		\[\begin{tikzcd}		
			C \arrow[hookrightarrow]{r}	\arrow[dr]  &M\\
			& F 	\arrow[u]&		
		\end{tikzcd}\]
\end{definition}
	
So every projective module is f-projective, and every finitely generated f-projective module is projective. Since flat modules are characterized as the direct limits of finitely generated free modules, each f-projective module is flat by the following lemma due to Lenzing, cf.\ \cite[Lemma 2.13]{GT}:

	\begin{lemma} \label{lem4}
		Let $ R $ be a ring and $ \mathcal{C} $ be a class of finitely presented modules closed under finite direct sums. Then the following are equivalent for a module $ M $.
		\begin{enumerate}
			\item Every homomorphism $ \varphi : G \rightarrow M $, where $ G $ is finitely presented, has a factorisation through a module in $ \mathcal{C} $.
			\item $ M \in \varinjlim \mathcal{C} $.
		\end{enumerate}
	\end{lemma}
		
The fact that f-projective modules are a particular kind of flat relative Mittag-Leffler modules goes back to Goodearl \cite{G}, see also \cite[Proposition 2.7]{F}:

\begin{proposition}\label{x1} 
		A module $ M $ is f-projective if and only if it is flat and $\{ R \}$-Mittag-Leffler. 
		
		In particular, each countably generated f-projective module is countably presented, and hence of projective dimension $\leq 1$.
	\end{proposition}
\begin{proof} Let $\mathcal F ^\prime$ denote the class of all flat left $R$-modules. By Lemma \ref{lim-ML} (or \cite[Theorem 1]{G}), $\mathcal D _{\{R \}} = \mathcal D _{\mathcal F ^\prime}$. By \cite[Theorem 1]{G}, for each module $M$, $M \in \mathcal D _{\mathcal F ^\prime}$, iff $M$ is flat and for each finitely generated submodule $F$ of $M$, the inclusion $F \hookrightarrow M$ factors through a finitely presented module. By Lemma \ref{lem4}, this is further equivalent to the f-projectivity of $M$.

The final claim follows from \cite[Corollary 5.3]{AH}.
\end{proof} 

We also note the following corollary of \cite[Theorem 5.1]{AH} (and Theorem \ref{one}):

\begin{corollary}\label{f-char}
Let $R$ be a ring and $M \in \rmod R$. Then $M$ is $f$-projective, if and only if $M$ possesses a system of submodules, $\mathcal U$, consisting of countably presented f-projective modules such that each countable subset of $M$ is contained in an element of $\mathcal U$ (and $\mathcal U$ is closed under unions of countable chains).
\end{corollary}

We will denote by $\stixcal{FP}$ the class of all f-projective modules. So $\stixcal{FP} = \mathcal D _{\{ R \}} = \mathcal D _{\mathcal F ^\prime}$. 

There is an interesting relation between f-projectivity and coherence. Here, we will call a module $M$ \textit{coherent}, if all its finitely generated submodules are finitely presented. (Thus a ring $R$ is right coherent, if the regular right module is coherent.) 

\begin{lemma}\label{cohf} Let $R$ be a ring.
\begin{enumerate}
\item Let $ M $ be a flat coherent module. Then $ M $ is f-projective.
\item The ring $ R $ is right coherent, iff $\stixcal{FP}$ coincides with the class of all flat coherent modules.
\end{enumerate}
\end{lemma}
\begin{proof}
(i) Since $M$ is flat, $ M = \varinjlim F_i $ where the modules $ F_i $ are finitely generated and free. If $ C $ is a finitely generated submodule of $ M $, then $ C $ is finitely presented, so the inclusion $ C \subseteq M $ factors through some $ F_i $ by Lemma \ref{lem4}, whence $ M $ is f-projective.

(ii) In view of part (i) and Lemma \ref{lem4}, in order to prove the only-if part, we have to prove that f-projectivity implies coherence for any module $M$. Let $ F $ be be a finitely generated submodule of $ M $. By f-projectivity, $ F $ is a submodule of a finitely generated free module. Since $ R $ is right coherent, $ F $ is finitely presented.

The if part is clear, since the regular module $ R $ is always f-projective.
\end{proof}

Note that the situation simplifies for coherent domains:

\begin{lemma}
	Let $ R $ be a coherent domain. Then $ \stixcal{FP} = \mathcal{F} $.
\end{lemma}
\begin{proof}
	Let $ M $ be a flat module. By Lemma \ref{cohf}(i), it suffices to prove that $ M $ is coherent. Let $ F $ be a finitely generated submodule of $ M $. Then $ M $ and $ F $ are torsion-free, so by a classic result of Cartan and Eilenberg \cite[16.1]{GT}, $ F $ embeds into a finitely generated free module. By the coherence of $ R $, $ F $ is finitely presented, proving that $ M $ is coherent.
\end{proof}

Further instances of the coincidence $ \stixcal{FP}$ with $\mathcal{F} $ (i.e., of the case when $\stixcal{FP}$ is a covering class, see Theorem \ref{two}) appear in part (i) of the following proposition:

	\begin{proposition}\label{observe} Let $R$ be a ring.
		\begin{enumerate}
			\item Assume that $ R $ is right noetherian or $R$ is right perfect. Then $ \stixcal{FP} = \mathcal{F} $ is a covering class.
 			\item If $ R $ is right non-singular, then all f-projective modules are non-singular.
			\item If $ R $ is von Neumann regular, then  $\stixcal{FP} = \mathcal F \mathcal M $. Hence $\stixcal{FP}$ is covering, only if $R$ is completely reducible.   
      \item Assume $ R $ is von Neumann regular and right self-injective. Then $ \stixcal{FP}$ coincides with the class of all non-singular modules.	
	  \end{enumerate}
 \end{proposition}
		\begin{proof} (i) If $R$ is right noetherian, then each finitely generated module is finitely presented, so Lemma \ref{lem4} yields that all flat modules are f-projective. If $R$ is right perfect, then all relative Mittag-Leffler modules are projective (= flat). \\
	(ii) Let $M$ be an f-projective module. By Definition \ref{f-proj}, each finitely generated submodule $C$ of $M$ embeds into a finitely generated free module $F$. By assumption, $F$ is non-singular, hence so are $C$ and $M$. \\
	(iii) This is clear, since von Neumann regularity of $R$ is equivalent to the property that all left $R$-modules are flat. \\
	(iv) The non-singularity of all f-projective modules follows from (ii). Since $R$ is right self-injective, \cite[9.2]{G1} shows that all finitely generated submodules of non-singular modules are projective, hence all non-singular modules are f-projective by Definition \ref{f-proj}. Alternatively, we can use (iii) and the fact that under the assumptions of (iv), flat Mittag-Leffler modules coincide with the non-singular ones by \cite[6.8]{HT}.
		\end{proof}
	
For semihereditary rings, we have a fully ring theoretic characterization:

	\begin{proposition}\label{semi} 
		\begin{enumerate}
    \item The following conditions are equivalent for a ring $R$:
			\begin{enumerate}
			\item[{(1)}] $R$ is right semihereditary.
			\item[{(2)}] $\stixcal{FP}$ coincides with the class of all modules $M$ such that each finitely generated submodule of M is projective.
			\item[{(3)}] The class $\stixcal{FP}$ is closed under submodules.
			\end{enumerate}
		\item Assume $R$ is right semihereditary. Then the following conditions are equivalent:
			\begin{enumerate}
			\item[{(1)}] $\stixcal{FP}$ is a covering class.
			\item[{(2)}] Each finitely generated flat module is projective.
			\item[{(3)}] For each $n > 0$, the full matrix ring $M_n(R)$ contains no infinite sets of non-zero pairwise orthogonal idempotents.
			\end{enumerate}
		\end{enumerate}
	\end{proposition}
		\begin{proof} (i) By \cite[Theorem 2.29]{L}, $R$ is right semihereditary, iff all finitely generated submodules of projective modules are projective. So the implication (1) implies (2) is immediate from Definition \ref{f-proj}, (2) implies (3) is trivial, and (3) implies (1) because projective modules are f-projective, and finitely generated f-projective modules are projective. 
		
		(ii) The implication (1) $ \Rightarrow $ (2) holds for any ring: By Theorem \ref{two}, (1) implies $\stixcal{FP} = \mathcal{F}$, so each finitely generated flat module is f-projective, hence projective, and (2) holds. If $R$ is right semihereditary, then $R$ has flat dimension $\leq 1$, i.e., submodules of flat modules are flat. So if each finitely generated flat module is projective, then by part (i), each flat module is f-projective. This proves (2) $ \Rightarrow $ (1).   
		
		Assume (2) holds, and there is an $n > 0$ such that the full matrix ring $S = M_n(R)$ contains an infinite set $\{ e_i \mid i < \omega \}$ of non-zero pairwise orthogonal idempotents. Then $M = S/\oplus \sum_{i < \omega} e_iS$ is a direct limit of the projective modules $S/\oplus \sum_{i \in X} e_iS$, where $X$ runs over all finite subsets of $\omega$. So $M$ is a cyclic flat right $S$-module which is not projective. Further, $S$ is Morita equivalent to $R$. If $(F,G)$ is a pair functors realizing this equivalence, then $F(M)$ is a finitely generated non-projective flat module, in contradiction with (2). 
		
		The implication (3) $ \Rightarrow $ (2) is proved e.g. in \cite[Proposition 4.10]{PR}.
	\end{proof}

On the one hand, if $ \stixcal{FP} \subsetneq \mathcal{F} $, then there is always a finitely generated projective module $M \in \stixcal{FP}$ such that $\varinjlim \add M \nsubseteq \stixcal{FP}$ (just take $M = R$). On the other hand, we have the following result that applies, e.g., to any simple projective module $M$: 

\begin{proposition}\label{closed} Let $R$ be a ring and $M$ be an f-projective module. Let $S = \End M$. Assume that $S$ is right noetherian and $M$ is a flat left $S$-module. Then $\varinjlim \add M \subseteq \stixcal{FP}$. 
\end{proposition}
\begin{proof} By \cite[Theorem 2.5]{PT}, $\varinjlim \add M = \{ F \otimes _S M \mid F \hbox{ a flat right $S$-module }\}$. So we have to prove that $F \otimes _S M$ is a flat and $\{ R \}$-Mittag-Leffler module, for each flat right $S$-module $F$. Flatness of $F \otimes _S M$ is clear since $M$ is f-projective, hence flat.

Let $I$ be a set. By assumption, the canonical map $M \otimes _R R^I \to M ^I$ is an injective homomorphism of left $S$-modules, whence the map 
$(F \otimes _S M) \otimes _R R^I \to F \otimes _S M ^I$ is monic. Since $S$ is right noetherian, $F$ is an f-projective right $S$-module by Proposition \ref{observe}(i). Since $M$ is a flat left $S$-module, the canonical map $F \otimes _S M ^I \to (F \otimes _S M)^I$ is monic by Proposition \ref{x1}. Thus, the module $F \otimes _S M$ is $\{ R \}$-Mittag-Leffler.                  
\end{proof}

\end{document}